\documentclass{article}
\usepackage{amsmath,amssymb}
\usepackage{graphicx}
\usepackage{amsthm}

\theoremstyle{definition}
\newtheorem{prop}{Proposition}
\newtheorem{thm}{Theorem}
\newtheorem{lem}{Lemma}

\begin{document}
\begin{center}
{\Large New Airy-type solutions of the ultradiscrete Painlev\'{e} II equation with parity variables}\medskip\\
Hikaru Igarashi$^*$, Shin Isojima$^\dagger$ and Kouichi Takemura$^*$\medskip\\
$^*$Department of Mathematics, Faculty of Science and Engineering, Chuo University, 1-13-27 Kasuga, Bunkyo-ku Tokyo 112-8551, Japan\medskip\\
$^\dag$
Department of Industrial and Systems Engineering, Faculty of Science and Engineering, Hosei University, 3-7-2 Kajino-cho, Koganei-shi, Tokyo 184-8584, Japan
\end{center}
\section*{Abstract}
The $q$-difference Painlev\'{e} II equation admits special solutions written in terms of determinant whose entries are the general solution of the $q$-Airy equation. An ultradiscrete limit of the special solutions is studied by the procedure of ultradiscretization with parity varialbes.
Then we obtain new Airy-type solutions of the ultradiscrete Painlev\'{e} II equation with parity variables, and the solutions have richer structure than the known solutions.
\section{Introduction}
The Painlev\'{e} equations and the discrete Painlev\'{e} equations play important roles in the areas of mathematical physics such as conformal field theory and random matrices.
They have been studied extensively from a viewpoint of the symmetry (see e.g. \cite{OGH,Sak0}). Especially, Hirota's direct method of soliton theory \cite{Hir} affected the discoveries of solutions in terms of determinants to the Painlev\'{e}-type equations. On the other hand, by a suitable limit (ultradiscrete limit \cite{TTMS}) from the $q$-difference Painlev\'{e} equations (one of a family of the discrete Painlev\'{ e} equations), some systems of cellular automaton, which we call the ultradiscrete Painlev\'{e} equations, may appear \cite{TTGOR}. Cellular automaton is a discrete dynamical system in which both dependent and independent variables take discrete values, and it is suitable for computer experiments. Moreover, it is free from numerical errors and sometimes admits exact special solutions.
The ultradiscrete Painlev\'{e} equations are expected to preserve the essential properties of the original equations. 

In this paper, we focus on the $q$-difference Painlev\'{e} II equation ($q$-PII),
\begin{equation}
(z(q\tau) z(\tau) + 1)(z(\tau) z(q^{-1}\tau) + 1) = \frac{a\tau^2z(\tau)}{\tau - z(\tau)}. \label{eq:qPII}
\end{equation}
The p-ultradiscrete analogues of the $q$-PII with $a=q^{2N+1}$ (p-ultradiscrete PII) and a class of its special solutions were investigated in \cite{IKMMS,IS,IST}. Here "p-ultradiscrete" is the abbreviation for "ultradiscrete with parity variables" \cite{MIMS09}, and it is crucial to introduce the parity variables to consider the ultradiscrete limit of solutions in terms of determinants.
The p-ultradiscrete PII is derived as follows (for details, see \cite{IKMMS}).  
We introduce a parameter $\varepsilon>0$ by
\begin{equation}
q=e^{Q/\varepsilon}, \quad Q<0, \label{eq:qtoQ}
\end{equation}
and the parity (or sign) variable $\zeta^{(N)}(m) \in \{+1,1\}$ and the amplitude variable $Z^{(N)}(m)$ by
\begin{equation}
\zeta^{(N)}(m) = \frac{z(q^m)}{|z(q^m)|},\quad |z(q^m)| = e^{Z^{(N)}(m)/\varepsilon},
\end{equation}
respectively.
Moreover, we define a function $s:\{1, -1\} \to \{0, 1\}$ by
\begin{equation}
s(\omega) =
\begin{cases}
1 & (\omega=1) \\
0 & (\omega=-1)
\end{cases}
\end{equation}
and rewrite $z(q^m)$ as
\begin{equation}
z(q^m)=\{s(\zeta^{(N)}(m)) - s(-\zeta^{(N)}(m))\} e^{Z^{(N)}(m)/\varepsilon}. \label{eq:ztozetaZ}
\end{equation}
We deform \eqref{eq:qPII} to a form without division by multiplying $q^m-z(q^m)$ on both sides and substitute \eqref{eq:ztozetaZ} into the resulting equation. 
Then, we collect non-negative terms to each side of equality, apply $\varepsilon \log$ to both sides and take the limit $\varepsilon \to +0$.
By applying the identity
\begin{align}
\lim_{\varepsilon \to +0} \varepsilon \log (s(\omega) e^{X/\varepsilon} + e^{Y/\varepsilon}) = \max (S(\omega) + X, Y),
\end{align}
where the function $S:\{1, -1\} \to \{0, -\infty\}$ is defined by 
\begin{align}
S (\omega) :=
\begin{cases}
0 & (\omega=+1) \\
-\infty & (\omega=-1),
\end{cases}
\end{align}
we obtain the p-ultradiscrete PII
\begin{align}
\max\Bigl[
&Z_{m+1} + 3Z_{m} + Z_{m-1} + S (\zeta_{m+1}\zeta_{m}\zeta_{m-1}), Z_{m+1} + 2Z_{m} + S (\zeta_{m+1}), \nonumber \\
&2Z_{m} + Z_{m-1} + S (\zeta_{m-1}), Z_{m} + S (\zeta_{m}), Z_{m} + (2N+1)Q + 2m Q + S (\zeta_{m}), \nonumber \\
&Z_{m+1} + 2Z_{m} + Z_{m-1} + mQ + S (-\zeta_{m+1}\zeta_{m-1}), \nonumber \\
&Z_{m+1} + Z_{m} + mQ + S (-\zeta_{m+1}\zeta_{m}), Z_{m} + Z_{m-1} + mQ + S (-\zeta_{m}\zeta_{m-1}) \Bigr] \nonumber \\
= \max\Bigl[
&Z_{m+1} + 3Z_{m} + Z_{m-1} + S (-\zeta_{m+1}\zeta_{m}\zeta_{m-1}), Z_{m+1} + 2Z_{m} + S (-\zeta_{m+1}), \nonumber \\
&2Z_{m} + Z_{m-1} + S (-\zeta_{m-1}), Z_{m} + S (-\zeta_{m}), Z_{m} + (2N+1)Q + 2m Q + S (-\zeta_{m}), \nonumber \\
&Z_{m+1} + 2Z_{m} + Z_{m-1} + mQ + S (\zeta_{m+1}\zeta_{m-1}), \nonumber \\
&Z_{m+1} + Z_{m} + mQ + S (\zeta_{m+1}\zeta_{m}), Z_{m} + Z_{m-1} + mQ + S (\zeta_{m}\zeta_{m-1}), mQ \Bigr], \label{eq:udP2}
\end{align}
where $Z_m = Z^{(N)}(m)$ and $\zeta_m = \zeta^{(N)}(m)$. Note that we have derived a simpler expression than that in \cite{IKMMS} by using formula in \cite{TT} for the function $s$;
\begin{align}
s(\zeta)s(\zeta ')+s(-\zeta)s(-\zeta ')=s(\zeta\zeta '),\ (s(\zeta)+s(-\zeta))^2 = 1.
\end{align}
Conversely, the function $\zeta_m e^{Z_m /\varepsilon}$ may approximate a solution of $q$-PII with $q = e^{Q/\varepsilon}$. 
In this paper we obtain several special solutions to p-ultradiscrete PII.
Some of them are described as follows;
\begin{thm}\label{thm0}
Assume that $Q<0$, $N \in \mathbb{Z} _{\ge 0} $, $m_0$ is an integer such that $ m_0\leq \min (-3N-2, -N(N+1)/2 -1) $ and a parameter $C$ satisfies $ - (m_0+ N(N+1))Q <C< (m_0 +1) Q$.
Then the following function $(\zeta^{(N)}(m) , Z^{(N)}(m)) $ is a solution to (\ref{eq:udP2}).\\
(I) If $m \leq m_0 - 2N -1$ or $m_0 + N +1 \leq m$, then 
\begin{equation}
(\zeta^{(N)}(m) , Z^{(N)}(m))
= \begin{cases}
(+1 , (-m-2N-1)Q) & (m \leq m_0 - 2N -1) \\
(+1, mQ) & (m_0 + N +1 \leq m\leq -1) \\
((-1)^m ,0 ) & (m\geq 0)
\end{cases}
\label{eq:thm0asym}
\end{equation}
(II) If $m_0 - 2N \leq m \leq m_0 + N $, then 
\begin{align}
& (\zeta^{(N)}(m) , Z^{(N)}(m))
= \nonumber \\
& \begin{cases}
((-1)^{j} , -C-j^2 Q) & (m=m_0 -2N + 3j) \\
(+1 , (m_0+j+1)Q) & (m=m_0 -2N + 3j +1) \\
((-1)^{j} , C+(j+1)^2 Q) & (m=m_0 -2N + 3j +2)  
\end{cases}
\label{eq:thm03cycle}
\end{align}
where $0 \leq j \leq N$ in the first case and $0 \leq j \leq N -1 $ in the second and the third cases.
\end{thm}
Let us explain our intention to obtain Theorem \ref{thm0}.
Hamamoto, Kajiwara and Witte \cite{HKW} established that \eqref{eq:qPII} admits a class of special solutions expressed in terms of determinants when $a=q^{2N+1}$ for $N\in\mathbb{Z}$. The elements of the determinants are written by a solution of the $q$-difference Airy equation, or shortly $q$-Airy equation, 
\begin{equation}
w (q \tau) - \tau w (\tau) + w (q^{-1} \tau) = 0. \label{eq:qAiryEq}
\end{equation}
We here call such a solution `seed'.
The  $q$-Airy equation has special solutions, the $q$-Ai function and the $q$-Bi function.
If we give the $q$-Ai function (resp. the $q$-Bi function) as a seed, we have special solutions of $q$-PII which we call the $q$Ai-type solutions (resp. the $q$Bi-type solutions).
In \cite{IST}, special solutions of (\ref{eq:udP2}) have been derived from the $q$Ai- or $q$Bi-type solutions through the limiting procedure. Hence, we call these ultradiscrete solutions as the uAi- and uBi-type solutions, respectively.
In this paper, we take linear combinations of the $q$-Ai and the $q$-Bi functions as seeds. 
Then we obtain the corresponding ultradiscrete solutions, which include Theorem \ref{thm0} as a special case, and they have richer structure than the uAi- and the uBi-type solutions. 
From a different point of view, our ultradiscrete solutions capture explicit behavior of solutions to $q$-PII in terms of determinants as $q \to 0$, and our results may help a understanding of behavior of solutions to $q$-PII.

This paper is organized as follows.  
In section \ref{sec:qAiry}, we review some results for \eqref{eq:qAiryEq} and study the p-ultradiscrete analogue of the general solution to \eqref{eq:qAiryEq}. 
In section \ref{sec3}, we review special solutions to $q$-PII in terms of determinants obtained in \cite{HKW}, and we investigate the p-ultradiscrete limit of these solutions.
As a consequence, we obtain explicit functional forms of the special solutions of p-ultradiscrete PII. 
In section \ref{concl}, we give concluding remarks.

Throughout this paper, we assume $Q<0$.
\section{$q$-Airy equation} \label{sec:qAiry}
We review some results for the $q$-Airy equation \eqref{eq:qAiryEq} and its p-ultradiscrete analogue.
We set $\tau=q^m$ ($m\in\mathbb{Z}$) and rewrite $w(q^m)$ as $w(m)$ for simplicity.
Then, the $q$-Ai and $q$-Bi functions
\begin{align}
a(m) &= q\mbox{-Ai} (q^m)= (-1)^{m(m-1)/2} \sum_{n=0}^{\infty} \frac{(-1)^n (-q)^{n(n+1)/2}}{(q^2;q^2)_n} (-q)^{mn} \\
b(m) &= q\mbox{-Bi} (q^m)=(-1)^{m(m+1)/2} \sum_{n=0}^{\infty} \frac{(-q)^{n(n+1)/2}}{(q^2;q^2)_n} (-q)^{mn},
\end{align}
are special solutions of \eqref{eq:qAiryEq}.
Here, $(q^2;q^2)_n$ is given by
\begin{equation}
(q^2;q^2)_n = 
\begin{cases}
1 & (n=0) \\
(1-q^2)(1-q^4)\cdots(1-q^{2n}) & (n=1,2,3,\ldots),
\end{cases}
\end{equation}
which is the $q$-shifted factorial. 
As $q\to 0$, these solutions are evaluated \cite{IST} as
\begin{align}
a(m) = 
\begin{cases}
(-1)^{m(m-1)/2} (1+O(q)) & (m \geq 0) \\
q^{m(m-1)/2} (1+O(q)) & (m \leq -1)
\end{cases}\label{eq:eval_a}\\
b(m) = 
\begin{cases}
(-1)^{m(m+1)/2} (1+O(q)) & (m \geq 0) \\
q^{-m(m+1)/2} (2+O(q)) & (m \leq -1),
\end{cases}\label{eq:eval_b}
\end{align}
respectively. The general solution for \eqref{eq:qAiryEq} is given by the linear combination
\begin{equation} \label{eq:w}
w(m) = c_1 a(m) + c_2 b(m).
\end{equation}
For ultradiscretization, we introduce a parameter $\varepsilon>0$ by \eqref{eq:qtoQ}.
We define the sign variable by $\omega_m=w(m)/|w(m)| \in \{1, -1\}$ and the amplitude variable by $|w(m)| = e^{W_m/\varepsilon}$.
Then $w(m)$ is written as
\begin{equation}
w(m)=\{s(\omega_m) - s(-\omega_m)\} e^{W_m/\varepsilon}.
\end{equation}
Following the procedure of p-ultradiscretization, we obtain the p-ultradiscrete Airy equation \cite{IKMMS}
\begin{align}
&\max \left(W_{m+1} + S (\omega_{m+1}),\ mQ + W_{m} + S (-\omega_{m}),\ W_{m-1} + S (\omega_{m-1}) \right) \nonumber\\
= &\max \left(W_{m+1} + S (-\omega_{m+1}),\ mQ + W_{m} + S (\omega_{m}),\ W_{m-1} + S (-\omega_{m-1}) \right). \label{eq:ud-Airy}
\end{align}

The p-ultradiscrete analogues of $a(m)$ and $b(m)$ have derived in \cite{IST} through the p-ultradiscrete limit.
We construct solutions of \eqref{eq:ud-Airy} from \eqref{eq:w}. The key is the following lemma which actually appears in \cite{IST}: 
\begin{lem} \label{lem1}
If a solution $w(m)$ for \eqref{eq:qAiryEq} is expanded into a series of $q=e^{Q/\varepsilon}$ as
\begin{equation}
w(m) = (-1)^{\hat{\omega}(m)}e^{C(m)/\varepsilon}q^{\hat{W}(m)}\sum_{k=0}^{\infty}d(k,m)q^k, \label{eq:expand}
\end{equation}
where $\hat{\omega}(m)$, $C(m)$, $\hat{W}(m)$, $d(k,m)$ and $Q(<0)$ are independent of $\varepsilon$ and $d(0,m)>0$.
Then, the pair of the sign variable $\omega_m=(-1)^{\hat{\omega}(m)}$ and the amplitude variable $W_m=\hat{W}(m) Q+C(m)$ solves \eqref{eq:ud-Airy}. $\blacksquare$
\end{lem}
Note that this Lemma can be applied to $q$-PII and its solutions.
We find from Lemma \ref{lem1} that the term $(-1)^{\hat{\omega}(m)}e^{C(m)/\varepsilon}q^{\hat{W}(m)}=(-1)^{\hat{\omega}(m)}e^{\{\hat{W}(m)Q+C(m)\}/\varepsilon}$, which has the lowest order of $q$ in the series \eqref{eq:expand}, is important for our study.  We call this term the \textit{leading term}.
If two functions $x(m)$ and $y(m)$ have the same leading term, we write $x(m) \sim y(m)$. For example, \eqref{eq:expand} is written as $w(m) \sim (-1)^{\hat{\omega}(m)}e^{C(m)/\varepsilon}q^{\hat{W}(m)}$. We define a term \textit{dominant} as follows. We consider some functions $x_k(m)$ $(k=1,2,\ldots,K)$ and write their leading terms as $|x_k(m)|\sim e^{\{\hat{W}_k(m)Q+C_k(m)\}/\varepsilon}$. 
For fixed $m$, $x_k(m)$ is dominant among $x_1(m),\ldots, x_K(m)$, if $\hat{W}_k(m)Q+C_k(m) > \hat{W}_l(m)Q+C_l(m)$ $(l=1,\ldots,k-1,k+1,\ldots,K)$ hold. 

We define the sign variables for the coefficients in \eqref{eq:w} as $\alpha=c_1/|c_1|$ and $\beta=c_2/|c_2|$, respectively, and put
\begin{equation}
c_1 = \alpha e^{A/\varepsilon},\ \ c_2 = \beta e^{B/\varepsilon}. \label{eq:c12toAB}
\end{equation}
Then, \eqref{eq:w} is evaluated as
\begin{equation}
w(m) =
\begin{cases}
\alpha(-1)^{m(m-1)/2} e^{A /\varepsilon} (1+O(q)) & (m\ge 0, A>B) \\
\beta (-1)^{m(m+1)2} e^{B/\varepsilon} (1+O(q))& (m\ge 0, A<B) \\
\alpha e^{A/\varepsilon} q^{m(m-1)/2} (1+O(q)) & (m\le -1, m^2<(B-A)/Q) \\
\beta e^{B/\varepsilon} q^{-m(m+1)/2} (2+O(q)) & (m\le -1, m^2\ge (B-A)/Q),
\end{cases}\label{eq:w_ev}
\end{equation}
(see \cite{IST}).
For simplicity, we consider the case $B-A<Q$ in this paper.
Then, there uniquely exists $m_0\in\mathbb{Z}_{\le -2}$ which satisfies 
\begin{align}
{m_0}^2Q \le B-A<(m_0+1)^2Q. \label{eq:m_0}
\end{align}
Hence, for $B-A<Q$, the p-ultradiscrete analogue of $w(m)$ is written as
\begin{equation}
(\omega_m, W_m) =
\begin{cases}
\left(\alpha (-1)^{m(m-1)/2}, A\right) & (m\ge 0) \\
\left(\alpha, m(m-1)Q/2+A\right) & \left( m_0 +1 \le m \le -1 \right) \\
\left(\beta, -m(m+1)Q/2+B\right) & \left(m \le m_0 \right) 
\end{cases}
\end{equation}
and it satisfies \eqref{eq:ud-Airy} by Lemma \ref{lem1}.
\section{$q$-Painlev\'{e} II equation}\label{sec3}
It is shown in \cite{HKW} that \eqref{eq:qPII} with $a=q^{2N+1}$ $(N\in\mathbb{Z})$ admits a class of special solutions. For later discussion, we review some results only for $N\ge 0$.
Equation \eqref{eq:qPII} with $a=q^{2N+1}$ is solved by
\begin{align}
z^{(N)} (\tau) &= 
\dfrac{g^{(N)} (\tau) g^{(N+1)} (q \tau)}{q^N g^{(N)} (q \tau) g^{(N+1)} (\tau)}\label{eq:gtoz}\\
g^{(N)} (\tau) &= 
\begin{cases}
\begin{vmatrix}
w(\tau) & w(q^2 \tau) & \cdots & w(q^{2N-2} \tau) \\
w(q^{-1} \tau) & w(q \tau) & \cdots & w(q^{2N-3} \tau) \\
\vdots & \vdots & \ddots & \vdots \\
w(q^{1-N} \tau) & w(q^{3-N} \tau) & \cdots & w(q^{N-1} \tau)
\end{vmatrix} & (N>0) \\
1 & (N=0),
\end{cases}\label{eq:g_det}
\end{align}
where $w(\tau)$ is a solution of \eqref{eq:qAiryEq}. The functions $g^{(N)} (\tau)$ satisfy the bilinear equations 
\begin{align}
&q^{2N} g^{(N+1)} (q^{-1} \tau) g^{(N)} (q^2 \tau) - q^{N} \tau g^{(N+1)} (\tau) g^{(N)} (q \tau) + g^{(N+1)} (q \tau) g^{(N)} (\tau) = 0 \\
&q^{2N} g^{(N+1)} (q^{-1} \tau) g^{(N)} (q \tau) - q^{2N} \tau g^{(N+1)} (\tau) g^{(N)} (\tau) + g^{(N+1)} (q \tau) g^{(N)} (q^{-1} \tau) = 0. 
\end{align}
Hereafter, we put $\tau=q^m$ and write $g^{(N)}(\tau) = g^{(N)}(m)$. 
The p-ultradiscrete limits of $g^{(N)}(m)$ for the special cases $w(m)=a(m)$ or $w(m)=b(m)$ have been investigated in \cite{IST}. We study the case that $w(m)$ is the general solution \eqref{eq:w}. We often consider the case of $m\ll -1$ for simplicity. However, we observe interesting structure of the solutions even in this case. 
\subsection{Evaluation for $g^{(N)}(m)$}
Set 
\begin{align}
\boldsymbol{a}_m = (a(m) ~~ a(m+2) ~~ \cdots ~~ a(m+2N-2))\\
\boldsymbol{b}_m = (b(m) ~~ b(m+2) ~~ \cdots ~~ b(m+2N-2)). \nonumber
\end{align}
By substituting \eqref{eq:w} into \eqref{eq:g_det}, we have
\begin{equation}
g^{(N)}(m) = 
\begin{vmatrix}
c_1 \boldsymbol{a}_m + c_2 \boldsymbol{b}_m \\
c_1 \boldsymbol{a}_{m-1} + c_2 \boldsymbol{b}_{m-1}\\
\vdots \\
c_1 \boldsymbol{a}_{m-N+1} + c_2 \boldsymbol{b}_{m-N+1}
\end{vmatrix}.\label{eq:gBYab}
\end{equation}
We expand \eqref{eq:gBYab} by employing multi-linearity of the determinant and introduce notation $g^{(N)}_{*\cdots *}(m)$ as 
\begin{align}
g^{(N)}(m)
&= c_1^N \left|
\begin{array}{c}
\boldsymbol{a}_m \\ \vdots \\ \boldsymbol{a}_{m-N+1} \\
\end{array}
\right| \nonumber \\
&+ c_1^{N-1} c_2 \left\{ \left|
\begin{array}{c}
\boldsymbol{a}_m \\ \vdots \\ \boldsymbol{a}_{m-N+2} \\ \boldsymbol{b}_{m-N+1} \\
\end{array}
\right| + \left|
\begin{array}{c}
\boldsymbol{a}_m \\ \vdots \\ \boldsymbol{a}_{m-N+3} \\ \boldsymbol{b}_{m-N+2} \\ \boldsymbol{a}_{m-N+1} \\
\end{array}
\right| + \cdots + \left|
\begin{array}{c}
\boldsymbol{b}_m \\ \boldsymbol{a}_{m-1} \\ \vdots\\ \boldsymbol{a}_{m-N+1} \\
\end{array}
\right| \right\} \nonumber\\
&+ c_1^{N-2} c_2^2 \left\{ \left|
\begin{array}{c}
\boldsymbol{a}_m \\ \vdots \\ \boldsymbol{a}_{m-N+3} \\ \boldsymbol{b}_{m-N+2} \\ \boldsymbol{b}_{m-N+1} \\
\end{array}
\right| + \left|
\begin{array}{c}
\boldsymbol{a}_m \\ \vdots \\ \boldsymbol{a}_{m-N+4} \\ \boldsymbol{b}_{m-N+3} \\ \boldsymbol{a}_{m-N+2} \\ \boldsymbol{b}_{m-N+1} \\
\end{array}
\right| + \cdots + \left|
\begin{array}{c}
\boldsymbol{b}_m \\ \boldsymbol{b}_{m-1} \\ \boldsymbol{a}_{m-2} \\ \vdots\\ \boldsymbol{a}_{m-N+1} \\
\end{array}
\right| \right\} \nonumber\\
&+ \cdots + c_2^N \left|
\begin{array}{c}
\boldsymbol{b}_m \\ \vdots \\ \boldsymbol{b}_{m-N+1} \\
\end{array}
\right| \\
&=: c_1^N g^{(N)}_{a\cdots  a}(m) +  c_1^{N-1} c_2 \left\{ g^{(N)}_{a\cdots  ab}(m) + g^{(N)}_{a\cdots  aba}(m) + \cdots + g^{(N)}_{ba\cdots  a}(m) \right\} \nonumber \\
&+ c_1^{N-2} c_2^2 \left\{ g^{(N)}_{a\cdots  abb}(m) + g^{(N)}_{a\cdots  abab}(m) + \cdots + g^{(N)}_{bba\cdots  a}(m) \right\} \nonumber\\
&+ \cdots + c_2^N g^{(N)}_{b\cdots b}(m). 
\end{align}
We consider only the case in which all $a$'s and $b$'s in \eqref{eq:gBYab} have negative arguments, i.e., $m\le -2N +1$. To calculate the leading term of each $g^{(N)}_{*\cdots *}(m)$, we present the following two propositions. 
\begin{prop}\label{prop1}
Consider $g^{(N)}_{*\cdots *}(m)$ with $N-k$ $\boldsymbol{a}$'s and $k$ $\boldsymbol{b}$'s. By rearranging rows, $g^{(N)}_{*\cdots *}(m)$ can be written as
\begin{equation}
g^{(N)}_{*\cdots *}(m) =
\pm \left|
\begin{array}{c}
\boldsymbol{a}_{m-n_1} \\ \vdots \\ \boldsymbol{a}_{m-n_{N-k}} \\ \boldsymbol{b}_{m-n_1'} \\ \vdots \\ \boldsymbol{b}_{m-n_k'} \\
\end{array}
\right|
= \pm \left|
\begin{array}{cc}
A_{N-k,k} & A_{N-k,N-k} \\ B_{k,k} & B_{k,N-k} \\ 
\end{array}
\right|,
\end{equation}
where $n_1<\cdots<n_{N-k}$ and $n_1'<\cdots<n_k'$. Then, for $m\le -2N +1$, the leading term of $g^{(N)}_{*\cdots *}(m)$ as $q\to 0$ is given by the product of the diagonal terms of $A_{N-k,N-k}$ and the anti-diagonal terms of $B_{k,k}$. 
\end{prop}
\begin{proof}
It is readily shown that $a(m)$ and $b(m)$ monotonically increases and decreases for $m<0$, respectively. Therefore, considering Laplace expansion of $g^{(N)}_{*\cdots *}(m)$, we find that $|A_{N-k,N-k}|\times |B_{k,k}|$ contains the term with the minimum order. Among the monomial in $A_{N-k,N-k}$, the product of the diagonal terms has the minimum order \cite{IST}. Similarly, the product of the anti-diagonal terms has the minimum order among the monomial in $|B_{k,k}|$. 
\end{proof}
\begin{prop}\label{prop2}
Among $g^{(N)}_{*\cdots *}(m)$ with $N-k$ $\boldsymbol{a}$'s and $k$ $\boldsymbol{b}$'s, $g^{(N)}_{a\cdots ab\cdots b}(m)$ is dominant for $m\le -2N+1$. 
\end{prop}

\begin{proof}
By rearranging rows, $g^{(N)}_{*\cdots *}(m)$ with $N-k$ $\boldsymbol{a}$'s and $k$ $\boldsymbol{b}$'s is deformed as, except for the sign, 
\begin{align}
\begin{vmatrix}
\boldsymbol{a}_{n-\sigma_1+1} \\
\vdots \\
\boldsymbol{a}_{n-\sigma_{N-k}+1} \\
\boldsymbol{b}_{n-\sigma_{N-k+1}+1} \\
\vdots \\
\boldsymbol{b}_{n-\sigma_{N}+1} 
\end{vmatrix},\label{eq:g_rearrange}
\end{align}
where $\sigma$ is an appropriate permutation which satisfies
\begin{align}
\sigma_1 <\sigma_{2}<\cdots<\sigma_{N-k},\ \sigma_{N-k+1}<\sigma_{N-k+2}<\cdots<\sigma_{N}. \label{eq:sigma_cons.2}
\end{align}
Our aim is to show that the permutation
\begin{align}
\sigma_i = i\ \ (i=1,2,\ldots, N) \label{eq:min_sigma}
\end{align}
makes the corresponding $g^{(N)}_{*\cdots *}(m)$ dominant among all of the permutations.
From Proposition \ref{prop1}, \eqref{eq:eval_a} and \eqref{eq:eval_b}, we find that the leading term of \eqref{eq:g_rearrange} is given by
\begin{align}
&\prod_{l=1}^{N-k} a(m-\sigma_l+2k-1+2l) \prod_{\lambda=1}^{k} b(m-\sigma_{N-k+\lambda}+2k-1-2\lambda) \nonumber \\
\sim &\prod_{l=1}^{N-k} q^{(m-\sigma_l+2k-1+2l)(m-\sigma_l+2k-2+2l)/2} \prod_{\lambda=1}^{k} 2q^{-(m-\sigma_{N-k+\lambda}+2k-2\lambda)(m-\sigma_{N-k+\lambda}+2k-3\lambda)/2} \nonumber \\
= & {}2^k q^{F+\sum_{l=1}^{N-k}\{\sigma_{l}^2-(2m+4k+4l-3)\sigma_l\} /2-\sum_{\lambda=1}^{k}\{\sigma_{N-k+\lambda}^2-(2m+4k-4\lambda-1)\sigma_{N-k+\lambda}\} /2}, \label{eq:prop2-1}
\end{align}
where we write the terms which are independent of $l$ or $\lambda$ as $F=F(N,k,n)$. Let us study the permutation $\sigma$ at which the exponent in \eqref{eq:prop2-1} reaches the minimum value. Noticing the first summation in the exponent in \eqref{eq:prop2-1}, we consider the quadratic
\begin{align}
f_l(\sigma) = \sigma\{\sigma-(2m+4k+4l-3)\}, 
\end{align}
where $m$ and $k$ are fixed. For $\sigma >0$, $f_l(\sigma) > f_{l'}(\sigma)$ holds if $l < l'$. Hence, $\sigma_l$ $(l=1,2,\ldots, N-k)$ should be smaller as $l$ is smaller. That is, $\sigma_i = i$ $(i=1,2,\ldots, N-k)$ is implied. From the second summation, we focus on 
\begin{align}
f_\lambda (\sigma) = -\sigma\{\sigma-(2m+4k-4\lambda-1)\}. 
\end{align}
For $\sigma >0$, we have $f_\lambda(\sigma) > f_{\lambda'}(\sigma)$ if $\lambda > \lambda'$. Therefore, $\sigma_\lambda$ $(\lambda=N-k+1,\ldots, N)$ should be smaller as $\lambda$ is larger. Then, $\sigma_i = i$ $(i=N-k+1,\ldots, N)$ is implied. Since these implication does not have contradiction, we find that \eqref{eq:min_sigma} achieves the minimum value of \eqref{eq:prop2-1}. 
\end{proof}

We study the p-ultradiscrete analogue of $g^{(N)}_{*\cdots *}(m)$. For simplicity, we write $g^{(N)}_{a\cdots ab\cdots b}(m)$ with $k$ $\boldsymbol{b}$'s as
\begin{equation}
g^{(N)}_{a\cdots ab\cdots b}(m) = g^{(N)}_{k}(m).
\end{equation}
For ultradiscretization, we put \eqref{eq:qtoQ} and \eqref{eq:c12toAB}. 
We introduce the sign variables $\gamma^{(N)}(m)$ and $\gamma^{(N)}_k(m)$ for $g^{(N)}(m)$ and $g^{(N)}_k(m)$, respectively. Then, we put
\begin{align}
g^{(N)}(m) = \{s(\gamma^{(N)}(m))-s(-\gamma^{(N)}(m))\}e^{G^{(N)}(m)/\varepsilon} \\
g^{(N)}_k(m) = \{s(\gamma^{(N)}_k(m))-s(-\gamma^{(N)}_k(m))\}e^{G^{(N)}_k(m)/\varepsilon}.
\end{align} 
One can evaluate $g^{(N)}(m)$ by using Propositions \ref{prop1} and \ref{prop2}. Then, applying Lemma \ref{lem1}, we have the following proposition;
\begin{prop}\label{thm1}
For $m \le -2N+1 $, The p-ultradiscrete analogue of $g^{(N)}(m)$ is given by
\begin{align}
(\gamma^{(N)}(m) , G^{(N)}(m)) = 
\begin{cases}
(\gamma_0^{(N)}(m) , G_0^{(N)}(m)) & (B-A < f_1^{(N)}(m)) \\
(\gamma_1^{(N)}(m) , G_1^{(N)}(m)) & (f_1^{(N)}(m) \le B-A < f_2^{(N)}(m)) \\
\vdots \\
(\gamma_k^{(N)}(m) , G_k^{(N)}(m)) & (f_k^{(N)}(m) \le B-A < f_{k+1}^{(N)}(m)) \\
\vdots \\
(\gamma_{N-1}^{(N)}(m) , G_{N-1}^{(N)}(m)) & (f_{N-1}^{(N)}(m) \le B-A < f_N^{(N)}(m)) \\
(\gamma_N^{(N)}(m) , G_N^{(N)}(m)) & (f_N^{(N)}(m) \le B-A ),
\end{cases}\label{eq:ud_G}
\end{align}
where
\begin{align}
\gamma_k^{(N)}(m) &= (-1)^{Nk-k(k+1)/2} \label{eq:gamma_k}\\
G_k^{(N)}(m)
&= (N-k)A + kB \nonumber\\
&+ \biggl[ \left( -k + \frac12 N \right)m^2 + \left\{-3k^2 + (2N+1)k + \frac{N(N-2)}{2} \right\}m \nonumber \\
&- \frac83k^3 + \left(2N+\frac32 \right)k^2 + \left(-N+\frac16\right)k + \frac{N(N-1)(N-2)}{6} \biggr] Q  \label{eq:G_k}\\
f_k^{(N)}(m) &= G_{k-1}^{(N)}(m) - G_k^{(N)}(m) - A + B \nonumber\\
&= \left\{ (m-(N-3k+2))^2 - (N-k)(N-k+1) \right\} Q \label{eq:f_k}
\end{align}
for $k = 0 , 1 , \cdots , N$. 
\end{prop}
\begin{proof}
We first derive the expression \eqref{eq:gamma_k} and \eqref{eq:G_k}. Employing Proposition \ref{prop1} and \ref{prop2}, and then using \eqref{eq:eval_a} and \eqref{eq:eval_b}, we obtain 
\begin{align}
&g^{(N)}(m) \nonumber \\
&\sim {c_1}^N g_0^{(N)}(m) + \cdots + {c_1}^{N-k} {c_2}^k g_k^{(N)}(m) + \cdots + {c_2}^N g_N^{(N)}(m) \nonumber\\
&\sim \displaystyle{ \sum_{k=0}^{N} \left\{ c_1^{N-k} c_2^k (-1)^{Nk-k(k+1)/2} \prod_{i=1}^{N-k} a(m+2k-1+i) \prod_{j=1}^k b(m-N-2+3j) \right\} } \nonumber\\
&\sim \sum_{k=0}^{N} \biggl\{ 2^k c_1^{N-k} c_2^k (-1)^{Nk-k(k+1)/2} \nonumber \\
&\phantom{\sim}\times q^{\sum_{i=1}^{N-k} (m+2k-1+i)(m+2k-2+i)/2 - \sum_{j=1}^k (m-N-1+3j)(m-N-2+3j)/2} \biggr\}, \label{eq:g_eval1}
\end{align}
where we regard $\displaystyle \prod_{i=1}^{0} = 1$. Note that the sign $(-1)^{Nk-k(k+1)/2}$, which is actually \eqref{eq:gamma_k}, is derived from 
\begin{equation}
\mathrm{sgn}\begin{pmatrix}
k+1 & k+2 & \cdots & N & k & k-1 & \cdots & 2 & 1 \\
1 & 2 & \cdots & \cdots & \cdots & \cdots & \cdots & N-1 & N
\end{pmatrix}
= (-1)^{Nk-k(k+1)/2}.
\end{equation}
Moreover, we substitute \eqref{eq:qtoQ} and \eqref{eq:c12toAB} into \eqref{eq:g_eval1} and write the resulting expression as $g^{(N)}(m)\sim \sum 2^k (-1)^{Nk-k(k+1)/2} e^{G_k^{(N)}(m)/\varepsilon}$. Then, we find by direct calculation that the explicit form of $G_k^{(N)}(m)$ is written as \eqref{eq:G_k}. 

Next, we study the condition that ${c_1}^{N-k} {c_2}^k g_k^{(N)}(m)$ becomes dominant for a fixed $k$. We present the following lemma. 
\begin{lem}\label{lem2}
We assume $m\le -2N+3$. For a fixed $k=0,1,2,\ldots, N$,
\begin{align}
&G_k^{(N)}(m) > G_{k'}^{(N)}(m)\ \ (k' = 0,1,2,\ldots, k-1, k+1,\ldots, N) \nonumber \\
\Leftrightarrow {}&G_k^{(N)}(m) > G_{k-1}^{(N)}(m) \; {\rm and } \; G_k^{(N)}(m) > G_{k+1}^{(N)}(m)
\end{align}
holds.
\end{lem}
\begin{proof}
It is readily shown that for $l = 1,\ldots, k$, 
\begin{align}
G_k &> G_{k-l} \Leftrightarrow B-A > F(l), \\
F(l) := &\biggl\{m^2+(6k-2N-3l-1)m+8k^2+(-4N-8l-3)k \nonumber\\
&+(2l+1)N+\frac83l^2+\frac32l-\frac16\biggr\}Q.
\end{align}
Noting that $m\le -2N+3$, we obtain 
\begin{align}
F(1)-F(l) = \{(3l-3)m + P_1(N,k,l)\}Q > 0\ \ (l=2,3,\ldots,k),
\end{align}
where we omit the explicit expression of $P_1(N,k,l)$, which is independent of $m$. 
Hence, we obtain $G_k > G_{k-1} \Rightarrow G_k > G_{k-l}$. In a similar manner, one can show $G_k > G_{k+1} \Rightarrow G_k > G_{k+l'}$ for $l'=1,2,\ldots,N-k$. 

The converse is trivial. 
\end{proof}
We return to the proof of Proposition \ref{thm1}. The term $\gamma_k^{(N)}(m) e^{G_k^{(N)}(m)/\varepsilon}$ becomes the leading term of $g^{(N)}(m)$ if $G_k^{(N)}(m) > G_{k'}^{(N)}(m)$ $(k' = 0 , 1 , \cdots , k-1 , k+1 , \cdots , N)$ hold. These conditions are reduced to two inequalities by Lemma \ref{lem2} and moreover summarized as $f_k^{(N)}(m) < B-A < f_{k+1}^{(N)}(m)$, if we introduce 
\begin{equation}
f_k^{(N)}(m) = G_{k-1}^{(N)}(m) - G_k^{(N)}(m) - A + B, 
 \end{equation} 
whose explicit form is exactly \eqref{eq:f_k}. 

In the case $B-A=f_k^{(N)}$, we have  $ G_{k-1}^{(N)}(m)= G_{k}^{(N)}(m)$ and $G_{k}^{(N)}(m) > G_{k'}^{(N)}(m)$ $(k' = 0 , 1 , \cdots , k-2 , k+1 , \cdots , N)$.
Hence 
\begin{align}
& g^{(N)}(m) \sim 2^{k-1} ( \gamma^{(N)}_{k-1} (m) +2 \gamma^{(N)}_{k} (m) ) e^{G_{k}^{(N)}(m)/ \varepsilon } .
\end{align}
Since $\gamma^{(N)}_{k-1} (m) , \gamma^{(N)}_{k} (m) \in \{ 1,-1\} $, we have
$(\gamma^{(N)}(m) , G^{(N)}(m)) = (\gamma_k^{(N)}(m) , G_k^{(N)}(m)) $ in the case  $B-A=f_k^{(N)}$.
\end{proof}
Note that $(\gamma_0^{(N)}(m) , G_0^{(N)}(m))$ and $(\gamma_N^{(N)}(m) , G_N^{(N)}(m))$ are the uAi-type and uBi-type solutions, respectively, which were presented in \cite{IS} and \cite{IST}. 

Proposition \ref{thm1} can be restated in a different form by introducing $m_0$ defined in \eqref{eq:m_0}.
For simplicity, we assume that the values of $A,B$ and $Q$ are chosen as $m_0$ satisfies
\begin{align}
m_0 \leq \min (-3N+1, -N(N-1)/2 -1) , \quad {m_0}^2Q \le B-A<(m_0+1)^2Q. \label{eq:m_0ineq}
\end{align}
Set
\begin{align}
P'_0 &:= (m_0 + 1)^2Q,  \\
P'_j &:= (m_0^2 - (N-j)(N-j+1))Q ~~ (j = 1 , \cdots , N).\nonumber
\end{align}
Then, there exists an integer $k_0 \in \{ 0 ,1 , \cdots , N-1 \}$ such that $P'_{k_0 +1} \le B-A < P'_{k_0}$ holds.

\begin{thm}\label{thm2}
Assume that we have $m_0$ and $k_0$ mentioned above for (suitably) assigned values of $A,B$ and $Q$. Then, for a given $m\le -2N+1$, $(\gamma^{(N)}(m) , G^{(N)}(m))$ is written in terms of $(\gamma^{(N)}_k(m) , G^{(N)}_k(m))$ given in Proposition \ref{thm1} as follows: 
\\
(I) When $m \leq m_0 - 2N + 2$, we have
\begin{equation}
(\gamma^{(N)}(m) , G^{(N)}(m))=(\gamma_N^{(N)}(m) , G_N^{(N)}(m)). 
\end{equation}
(II) When $m_0 - 2N + 3 \leq m \leq m_0 + N - 1$, the value $m$ is written by introducing $j$ as 
\begin{equation}
m=
\begin{cases}
m_0 + N - 3j , m_0 + N - 3j + 1 , m_0 + N - 3j + 2  & (N-1 \geq j \geq k_0+1 ) \\
m_0 + N - 3j , m_0 + N - 3j + 1   & (j =k_0 ) \\
m_0 + N - 3j-1 , m_0 + N - 3j  , m_0 + N - 3j + 1  & (k_0 -1 \geq j \geq 0)
\end{cases}
\end{equation}
Then we have
\begin{equation}
(\gamma^{(N)}(m) , G^{(N)}(m))=(\gamma_j^{(N)}(m) , G_j^{(N)}(m)). 
\end{equation}
(III) When $m_0 + N \leq m \leq -2N+1 $, we have
\begin{equation}
(\gamma^{(N)}(m) , G^{(N)}(m))=(\gamma_0^{(N)}(m) , G_0^{(N)}(m)).
\end{equation}
\end{thm}
\begin{proof}
We show two typical cases only. Note that
\begin{align}
P'_N < P'_{N-1} < \cdots < P'_{k_0+1} < P'_{k_0} < P'_{k_0-1} < \cdots < P'_1 < P'_0 \label{eq:P_mt}
\end{align}
holds. In case (I), we have 
\begin{equation}
f^{(N)}_N(m) = (m+2N-2)^2Q \leq m_0^2Q \le B-A
\end{equation}
and therefore obtain $(\gamma^{(N)}(m) , G^{(N)}(m))=(\gamma_N^{(N)}(m) , G_N^{(N)}(m))$ from Proposition \ref{thm1}. Next, in case (II) with $m = m_0 + N - 3j$ $(N-1 \geq j \geq k_0+1 )$, we have $f^{(N)}_j(m) \le B-A < f^{(N)}_{j+1}(m)$ since 
\begin{align}
f_j &= \{(m_0-2)^2 - (N-j)(N-j+1)\}Q < m_0^2Q \le B-A, \\
f_{j+1} &= \{(m_0+1)^2 - (N-j-1)(N-j)\}Q > (m_0+1)^2Q > B-A.
\end{align}
Hence, from Proposition \ref{thm1}, we obtain $(\gamma^{(N)}(m) , G^{(N)}(m))=(\gamma_j^{(N)}(m) , G_j^{(N)}(m))$. The case of $m = m_0 + N - 3j + 1$ is shown by replacing $m_0$ in the above with $m_0+1$. For $m = m_0 + N - 3j + 2$, we have 
\begin{align} 
f_j &= \{m_0^2 - (N-j)(N-j+1)\}Q = P'_j \leq P'_{k_0+1} \le B-A, \\
f_{j+1} &= \{(m_0+3)^2 - (N-j-1)(N-j)\}Q > (m_0+1)^2Q > B-A.
\end{align}
Note that we have used \eqref{eq:P_mt} to show the first inequality. Then, $f_j \le B-A < f_{j+1}$ holds and we obtain $(\gamma^{(N)}(m) , G^{(N)}(m))=(\gamma_j^{(N)}(m) , G_j^{(N)}(m))$ from Proposition \ref{thm1}. 

All other cases can be shown in a similar manner. 
\end{proof}
\subsection{Evaluation for $z^{(N)}(m)$}
In Introduction, we gave the p-ultradiscrete analogue of \eqref{eq:qPII} with $\tau=q^m$ and $a=q^{2N+1}$ as \eqref{eq:udP2}.
The variable transformation from $g^{(N)}(m)$ to $z^{(N)}(m)$ was given in \eqref{eq:gtoz}. We may write the p-ultradiscrete analogue of \eqref{eq:gtoz} as follows:
\begin{align}
\begin{cases}
\zeta^{(N)}(m) = \gamma^{(N+1)}(m+1) \gamma^{(N+1)}(m) \gamma^{(N)}(m+1) \gamma^{(N)}(m) \\
Z^{(N)}(m) = G^{(N+1)}(m+1) - G^{(N+1)}(m) - G^{(N)}(m+1) + G^{(N)}(m) - NQ.
\end{cases}\label{eq:GtoZ}
\end{align}
We study special solutions $(\zeta^{(N)}(m), Z^{(N)}(m))$ constructed by this transformation from $(\gamma^{(N)}(m), G^{(N)}(m))$, $(\gamma^{(N)}(m+1), G^{(N)}(m+1))$, $(\gamma^{(N+1)}(m), G^{(N+1)}(m))$ and $(\gamma^{(N+1)}(m+1), G^{(N+1)}(m+1))$, which we obtain in the previous subsection. For this purpose, we present the following lemma.
\begin{lem}\label{lem3}
Assume that $m\le -2N-1$.
Set
\begin{align}
& h_{{\rm I},l}^{(N)} (m) = f_l^{(N+1)}(m) ,\;  h_{{\rm II},l}^{(N)} (m) = f_l^{(N)}(m) , \nonumber \\
& h_{{\rm III},l}^{(N)} (m) = f_l^{(N+1)}(m+1) , \; h_{{\rm IV},l}^{(N)} (m) = f_l^{(N)}(m+1) ,
\end{align}
where $f_l^{(N)}(m)$ was defined by \eqref{eq:f_k}.
Then we have the inequalities
\begin{equation}
 h_{{\rm I},l}^{(N)} (m) < h_{{\rm II},l}^{(N)} (m) < h_{{\rm III},l}^{(N)} (m) < h_{{\rm IV},l}^{(N)} (m) < h_{{\rm I},l+1} ^{(N)}(m) 
\label{eq:f_ineq.}
\end{equation}
for $l = 1,\cdots, N$ and 
\begin{equation}
h_{{\rm I},N+1}^{(N)} (m) < h_{{\rm III},N+1}^{(N)} (m) .
\end{equation}
\end{lem}
\begin{proof}
Each inequality is obtained from \eqref{eq:f_k} through direct computation. 
\end{proof}

Since the pair $(\zeta, Z)$ is constructed from four pairs of $(\gamma, G)$ by \eqref{eq:GtoZ}, we have to know explicit functional forms of each $\gamma$ and $G$. As an example, we consider \eqref{eq:GtoZ} when $B-A$ satisfies $ h_{{\rm I},l}^{(N)} (m)  \le B-A <  h_{{\rm II},l}^{(N)} (m) $ for an integer $l$. In this case, we find from \eqref{eq:f_ineq.} that 
\begin{align}
\begin{cases}
f_{l-1}^{(N)}(m) \le B-A < f_l^{(N)}(m) \\
f_{l-1}^{(N+1)}(m+1) \le B-A < f_l^{(N+1)}(m+1) \\
f_{l}^{(N+1)}(m) \le B-A < f_{l+1}^{(N+1)}(m) \\
f_{l-1}^{(N)}(m+1) \le B-A < f_l^{(N)}(m+1)
\end{cases}
\end{align}
hold. From these inequalities, we obtain 
\begin{align}
\begin{cases}
(\gamma^{(N)}(m), G^{(N)}(m))=(\gamma_{l-1}^{(N)}(m), G_{l-1}^{(N)}(m)) \\
(\gamma^{(N+1)}(m+1), G^{(N+1)}(m+1))=(\gamma_{l-1}^{(N+1)}(m+1), G_{l-1}^{(N+1)}(m+1)) \\
(\gamma^{(N+1)}(m), G^{(N+1)}(m))=(\gamma_{l}^{(N+1)}(m), G_{l}^{(N+1)}(m)) \\
(\gamma^{(N)}(m+1), G^{(N)}(m+1))=(\gamma_{l-1}^{(N)}(m+1), G_{l-1}^{(N)}(m+1)),
\end{cases}
\end{align}
respectively, through Proposition \ref{thm1}. Therefore, from \eqref{eq:GtoZ}, we have a special solution of \eqref{eq:udP2}, 
\begin{align}
\zeta^{(N)} 
&= \gamma_{l-1}^{(N+1)}(m+1) \gamma_l^{(N+1)}(m) \gamma_{l-1}^{(N)}(m+1) \gamma_{l-1}^{(N)}(m) \nonumber\\
&= (-1)^{-2l^2+4lN+3l-3N-1} \nonumber\\
&= (-1)^{l+N+1} \\
Z^{(N)}
&= G_{l-1}^{(N+1)}(m+1) - G_l^{(N+1)}(m) - G_{l-1}^{(N)}(m+1) + G_{l-1}^{(N)}(m) - NQ \nonumber\\
&= A - B + (m^2+(6l-2N-5)m+(8l^2+(-4N-13)l+3N+5))Q.
\end{align}
We find from this example that if we know a suitable condition for the value of $B-A$, we find the explicit expression of each $(\gamma, G)$ by Proposition \ref{thm1} and furthermore that of $(\zeta, Z)$ by \eqref{eq:GtoZ}.
By investigating the other cases, we have the following proposition.

\begin{prop}\label{thm3}
For $m\le -2N-1$, we have
\begin{align}
(\zeta^{(N)}(m) , Z^{(N)}(m))= 
\begin{cases}
(+1, mQ) , & (B-A < h_{{\rm IV},0}^{(N)}(m) ) \\
(\zeta_{{\rm I},l}^{(N)}(m) , Z_{{\rm I},l}^{(N)}(m)) & (h_{{\rm I},l}^{(N)}(m) \le B-A < h_{{\rm II},l}^{(N)}(m) ) \\
(\zeta_{{\rm II},l}^{(N)}(m) , Z_{{\rm II},l}^{(N)}(m)) & (h_{{\rm II},l}^{(N)}(m) \le B-A < h_{{\rm III},l}^{(N)}(m) ) \\
(\zeta_{{\rm III},l}^{(N)}(m) , Z_{{\rm III},l}^{(N)}(m)) & (h_{{\rm III},l}^{(N)}(m) \le B-A < h_{{\rm IV},l}^{(N)}(m) ) \\
(\zeta_{{\rm IV},l}^{(N)}(m) , Z_{{\rm IV},l}^{(N)}(m)) & (h_{{\rm IV},l}^{(N)}(m) \le B-A < h_{{\rm I},l+1}^{(N)}(m) ) \\
(\zeta_{{\rm I},N+1}^{(N)}(m) , Z_{{\rm I},N+1}^{(N)}(m)) & (h_{{\rm I},N+1}^{(N)}(m) \le B-A < h_{{\rm III},N+1}^{(N)}(m) ) \\
(+1, (-m-2N-1)Q) & (h_{{\rm III},N+1}^{(N)}(m) \le B-A ) \\
\end{cases}
\end{align}
where $l=1,2,\dots ,N$ and 
\begin{align}
&\begin{cases}
\zeta_{{\rm I},l}^{(N)}(m) = (-1)^{N+l+1}, \\
Z_{{\rm I},l}^{(N)}(m) = A - B + [m^2+(6l-2N-5)m +8 l ^2-(4N+13)l+3N+5]Q, 
\end{cases}\\
&\begin{cases}
\zeta_{{\rm II},l}^{(N)}(m) = -1, \\
Z_{{\rm II},l}^{(N)}(m) = (-m-2l+1)Q, 
\end{cases}\\
&\begin{cases}
\zeta_{{\rm III},l}^{(N)}(m) = (-1)^{N+l}, \\
Z_{{\rm III},l}^{(N)}(m) = -Z_{{\rm I},l}^{(N)}(m+1), 
\end{cases}\\
&\begin{cases}
\zeta_{{\rm IV},l}^{(N)}(m) = 1, \\
Z_{{\rm IV},l}^{(N)}(m) = (m+2l)Q. 
\end{cases}
\end{align}
\end{prop}

This proposition can be restated by introducing $m_0$ as Proposition \ref{thm1} is done. 
We assume that the values of $A,B$ and $Q$ are chosen as $m_0$ satisfies 
\begin{align}
m_0 \leq \min (-3N-2, -N(N+1)/2 -1) , \quad {m_0}^2Q \le B-A<(m_0+1)^2Q. \label{eq:m_0ineq2}
\end{align}
Set
\begin{align}
P_{0} &:= (m_0 + 1)^2Q, \\
P_j &:= (m_0^2 - (N-j+1)(N-j+2))Q ~~ (j = 1 , \cdots , N+1). \nonumber
\end{align}
Then, there exists an integer $k_0 \in \{ 0 , 1 , \cdots , N \}$ such that $P_{k_0 +1} \le B-A < P_{k_0}$ holds.
Using these notation, the result is written as follows:
\begin{thm} \label{thm4}
Assume that we have $m_0$ and $k_0 $ mentioned above for assigned values of $A,B$ and $Q$.
Then the following function $(\zeta ^{(N)}(m) , Z ^{(N)}(m))$ is obtained by the p-ultradiscrete limit of the solution of $q$-PII in terms of determinants.\\
(I) If $m \leq m_0 - 2N -1$, then 
\begin{equation}
(\zeta^{(N)}(m) , Z^{(N)}(m))
=
(+1 , (-m-2N-1)Q) 
\end{equation}
(II) If $m_0 - 2N \leq m \leq m_0 + N - 3k_0 +1 $, then 
\begin{align}
& (\zeta^{(N)}(m) , Z^{(N)}(m))
= \nonumber \\
& \begin{cases}
((-1)^{j} , A-B+(m_0^2+m_0-j^2)Q) & (m=m_0 -2N + 3j) \\
(+1 , (m_0+j+1)Q) & (m=m_0 -2N + 3j +1) \\
((-1)^{j} , B-A-(m_0^2+m_0-(j+1)^2)Q) & (m=m_0 -2N + 3j +2)  
\end{cases}
\end{align}
where $0 \leq j \leq N- k_0 $ in the first and the second cases and $0 \leq j \leq N- k_0 -1 $ in the third case.\\
(III) If $m = m_0 + N - 3k_0 +2 $ and $k_0 \neq 0$, then 
\begin{equation}
(\zeta^{(N)}(m) , Z^{(N)}(m))
= (-1 , (-m_0-N+k_0 -1)Q) 
\end{equation}
(IV) If $m_0 + N - 3k_0 +3 \leq m \leq m_0 + N $ and $k_0 \neq 0$, then
\begin{align}
& (\zeta^{(N)}(m) , Z^{(N)}(m))
= \nonumber \\
&  \begin{cases}
(+1 , (m_0+j)Q) & (m=m_0 -2N +3j) \\
((-1)^{j} ,B-A-(m_0^2-m_0-(j+1)^2)Q) & (m=m_0 -2N +3j+1)  \\
((-1)^{j+1} , A-B+(m_0^2-m_0-(j+1)^2)Q) & (m=m_0 -2N +3j+2) 
\end{cases}
\end{align}
where $N -k_0 +1 \leq j \leq N $ in the first case and $N- k_0 +1 \leq j \leq N-1 $ in the second and the third cases.\\
(V) If $m_0 + N +1 \leq m \leq -2N-1$, then
\begin{align}
(\zeta^{(N)}(m) , Z^{(N)}(m))
=(+1 , mQ) .
\end{align}
\end{thm}
\begin{proof}
In a similar manner to Theorem \ref{thm2}, it follows that
\begin{align}
&(\zeta^{(N)}(m) , Z^{(N)}(m))\nonumber\\
&=
\begin{cases}
(+1 , (-m-2N-1)Q)  & (m \leq m_0 - 2N - 1) \\
(\zeta_{{\rm I},N+1}^{(N)}(m) , Z_{{\rm I},N+1}^{(N)}(m)) & (m = m_0 - 2N) \\
(\zeta_{{\rm IV},j}^{(N)}(m) , Z_{{\rm IV},j}^{(N)}(m)) & (m = m_0 + N - 3j + 1, k_0+1 \leq j \leq N) \\
(\zeta_{{\rm III},j}^{(N)}(m) , Z_{{\rm III},j}^{(N)}(m)) & (m = m_0 + N - 3j + 2, k_0+1 \leq j \leq N) \\
(\zeta_{{\rm I},j}^{(N)}(m) , Z_{{\rm I},j}^{(N)}(m)) & (m = m_0 + N - 3j + 3, k_0+1 \leq j \leq N)\\
(\zeta_{{\rm IV},j}^{(N)}(m) , Z_{{\rm IV},j}^{(N)}(m)) & (m = m_0 + N - 3k_0 + 1 ,j=k_0) \\
(\zeta_{{\rm II},j}^{(N)}(m) , Z_{{\rm II},j}^{(N)}(m)) & (m = m_0 + N - 3k_0 + 2,j=k_0)\\
(\zeta_{{\rm IV},j-1}^{(N)}(m) , Z_{{\rm IV},j-1}^{(N)}(m)) & (m = m_0 + N - 3k_0 + 3,j=k_0) \\
(\zeta_{{\rm III},j}^{(N)}(m) , Z_{{\rm III},j}^{(N)}(m)) & (m = m_0 + N - 3j + 1, 1 \leq j \leq k_0 -1) \\
(\zeta_{{\rm I},j}^{(N)}(m) , Z_{{\rm I},j}^{(N)}(m)) & (m = m_0 + N - 3j + 2, 1 \leq j \leq k_0 -1) \\
(\zeta_{{\rm IV},j-1}^{(N)}(m) , Z_{{\rm IV},j-1}^{(N)}(m)) & (m = m_0 + N - 3j + 3, 1 \leq j \leq k_0 -1) \\
(+1 , mQ)  & (m \geq m_0 + N + 1)
\end{cases}
\end{align}
Hence we obtain the theorem by arranging the terms and replacing the value $j$ to $N-j$.
\end{proof}
Note that the function $(\zeta ^{(N)}(m) , Z ^{(N)}(m))$ in Theorem \ref{thm4} is a solution of p-ultradiscrete PII, which follows from Lemma \ref{lem1}.

\subsection{Solutions of p-ultradiscrete Painlev\'e II}

In \cite{IST}, two solutions of p-ultradiscrete PII were derived by choosing the $q$-Ai or $q$-Bi functions as the seed function $w(\tau)$. 
They are described as follows.\\
Ai-type solution:
\begin{equation}
(\zeta^{(N)}(m) , Z^{(N)}(m)) =
\left\{
\begin{array}{ll}
((-1)^m ,0 ) & (m\geq 0)\\
(+1, mQ) &  (m\leq -1).
\end{array}
\right.
\label{eq:Ai-type}
\end{equation}
Bi-type solution:
\begin{equation}
(\zeta^{(N)}(m) , Z^{(N)}(m)) =
\left\{
\begin{array}{ll}
((-1)^{m+1},0) & (m\geq -2N)\\
(+1,(-m-2N-1)Q )&  (m\leq -2N-1).
\end{array}
\right. 
\label{eq:Bi-type}
\end{equation}
On the other hand, the function in Theorem \ref{thm4} includes the Ai-type solution and the Bi-type solution in special ranges of $m$, and they are interpolated by specific functions.
The solutions in Theorem \ref{thm4} for $m \leq -2N-1$ can be combined with the Ai-type solution for $m \leq -2N-1$.
Namely we have the following theorem.
\begin{thm} \label{thm:combined}
Under the assumption of Theorem \ref{thm4}, the function $ (\zeta^{(N)}(m) , Z^{(N)}(m)) $ defined by Theorem \ref{thm4} for $m\leq m_0+N$ and 
\begin{equation}
(\zeta^{(N)}(m) , Z^{(N)}(m)) 
 =
\left\{
\begin{array}{ll}
(+1, mQ) & (m_0+N+1 \leq m \leq -1) \\
((-1)^m ,0 ) & (m\geq 0)
\end{array}
\right.
\label{eq:typeAi}
\end{equation}
is a solution to p-ultradiscrete PII \eqref{eq:udP2}.
\end{thm} 
\begin{proof}
p-ultradiscrete PII \eqref{eq:udP2} is a three term relation.
If $m \leq -2N-2$, then the relation \eqref{eq:udP2} is satisfied by Theorem \ref{thm4}.
If $m\geq m_0+N+2$, then the relation \eqref{eq:udP2} is satisfied because the Ai-type solution satisfies p-ultradiscrete PII.
Since $m_0 \leq -3N-2$, it remains to show the case $m_0=-3N-2$ and $m=-2N-1$, and it is shown directly that the function satisfies \eqref{eq:udP2} for $m=-2N-1$.
\end{proof}
We conjecture that the solution to $q$-PII which is ultradiscretized as the function in Theorem \ref{thm4} for $m \geq -2N-1$ is also ultradiscretized as \eqref{eq:typeAi}.

Theorem \ref{thm0} in Introduction is a special case of Theorem \ref{thm:combined}.
In fact, the case $k_0=0$ and $C= B-A-(m_0^2+m_0)Q$ in Theorem \ref{thm:combined} corresponds to Theorem \ref{thm0}.
We can also express the solutions for the case $k_0 \neq 0$ in a similar form to Theorem \ref{thm0}.
The solution in Theorem \ref{thm0} has richer structure than \eqref{eq:Ai-type} and \eqref{eq:Bi-type}. Although its asymptotic form \eqref{eq:thm0asym} is agree with \eqref{eq:Ai-type} or \eqref{eq:Bi-type}, \eqref{eq:thm03cycle} shows a structure of three-term cycles, which does not appear in \eqref{eq:Ai-type} and \eqref{eq:Bi-type}.
Moreover the solution has parameters $m_0$ and $C$ unlike \eqref{eq:Ai-type} and \eqref{eq:Bi-type}.

We now give an example of solutions of \eqref{eq:udP2}.
We consider the case of $A=450$, $B=25$, $Q=-3$ and $N=3$. Then, we have $m_0=-12$ and $k_0$ in Theorem \ref{thm4} becomes $k_0=2$. Applying Theorems \ref{thm4} and \ref{thm:combined}, we obtain a special solution of \eqref{eq:udP2},
\begin{align}
(\zeta^{(3)}(m), Z^{(3)}(m))=
\begin{cases}
(+1,3m+21) & (m\le -19) \\
(+1,29) & (m=-18) \\
(+1,33) & (m=-17) \\
(+1,-32) & (m=-16) \\
(-1,32) & (m=-15)\\
(+1,30) & (m=-14) \\
(-1,-30) & (m=-13) \\
(+1,30) & (m=-12) \\
(+1,16) & (m=-11) \\
(-1,-16) & (m=-10) \\
(+1,-3m) & (-9 \le m \le -1)\\
((-1)^m ,0 ) & (m\geq 0).
\end{cases}
\label{eq:z3}
\end{align}

We compare the solution \eqref{eq:z3} of p-ultradiscrete PII with a numerical solution of $q$-PII (see Figure \ref{fig1}).
\begin{figure}[h!tbp]
\centering\includegraphics[width=.7\linewidth]{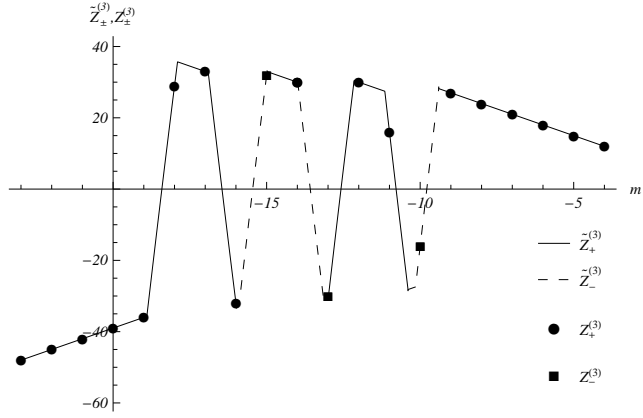}
\caption{Special solutions of $q$-PII and p-ultradiscrete PII.}
\label{fig1}
\end{figure}

In Figure \ref{fig1}, the solid lines (resp. the dashed lines)  represent the values of the function $\tilde{Z}^{(3)}_+ (m)$ (resp. $\tilde{Z}^{(3)}_- (m)$), where $z^{(3)} (m)  = e^{\tilde{Z}^{(3)}_+ (m)/\varepsilon }  $ (resp. $z^{(3)} (m) =- e^{\tilde{Z}^{(3)}_- (m) /\varepsilon } $ )  and $z^{(3)} (m)$ is the solution of $q$-PII  with $q=e^{Q/\varepsilon }$, $\varepsilon =0.1$, $Q=-3$, $N=3$, $A=450$ and $B=25$.
We write the values of $Z^{(3)}(m) $ with $\zeta ^{(3)} (m)=1$ (resp. $\zeta ^{(3)} (m)=-1$) by $Z^{(3)}_+(m) $ (resp. $Z^{(3)}_- (m) $) in \eqref{eq:z3} and plot it by the circles (resp. the squares) in Figure \ref{fig1}.
In the case $m \le m_0-2N-1$ or $m \ge m_0+N+1$, the solution \eqref{eq:z3} resembles the numerical solution of $q$-PII very much.
In the case $m_0 -2N \le m \le m_0+N$, the solution \eqref{eq:z3} grasps a feature of the solution of $q$-PII roughly, because we are handling \eqref{eq:z3} as a integer-valued function.
\section{Concluding Remarks} \label{concl}
We have given the p-ultradiscrete analogue of the general solution $w(m)$ for the $q$-Airy equation. Using this result, we have constructed the p-ultradiscrete analogues of special solutions for $q$-PII which have $w(m)$ as the seed.
As expected, the obtained solutions have richer structure than the $q$Ai- and $q$Bi-type solutions. We have found that their behavior is similar to that of a solution of $q$-PII.
Through the p-ultradiscretization, we clarify the asymptotics of the solutions of $q$-PII or p-ultradiscrete PII, although it is not clear from the expression in the form of determinant.
We expect that results in p-ultradiscrete PII are useful to analyze the property of $q$-PII. 
A technical problem is to consider the case that the condition $m_0 \leq \min (-3N-2, -N(N+1)/2 -1) $ is not satisfied.

Another problem for solutions of p-ultradiscrete PII is to investigate solutions which are not coming from special solutions in the form of determinant.
Note that Murata \cite{Mu} obtained some exact solutions to ultradiscrete PII without parity variable.

It is known that other $q$-difference Painlev\'e equations also admit solutions in the form of determinant.
Our results should be extended to the cases of other $q$-difference Painlev\'e equations.

\section*{Acknowledgment}
This work is supported by JSPS KAKENHI Grant Numbers 26790082 and 26400122.

\end{document}